\documentclass[10pt]{amsart}

\usepackage{amsfonts,amssymb,amsmath,amscd,amstext}
\usepackage[colorlinks=true,linkcolor=blue,citecolor=blue]{hyperref}
\usepackage[utf8]{inputenc}
\usepackage{microtype}
\usepackage{graphicx}
\usepackage{changes}
\usepackage{comment}

\renewcommand{\leq}{\leqslant}

\renewcommand{\le}{\leqslant}
\renewcommand{\ge}{\geqslant}
\newcommand{\ptl}{\partial}
\newcommand{\hhh}{{\mathcal{H}}}
\newcommand{\norm}[1]{|| #1 ||}

\newcommand{\rr}{{\mathbb{R}}}

\newcommand{\hh}{{\mathbb{H}}}

\newcommand{\sph}{{\mathbb{S}}}
\newcommand{\la}{\lambda}

\newcommand{\sg}{\sigma}
\newcommand{\Om}{\Omega}
\newcommand{\eps}{\varepsilon}

\newcommand{\ga}{\gamma}
\newcommand{\Ga}{\Gamma}

\newcommand{\escpr}[1]{\langle#1\rangle}

\newcommand{\mh}{\mathcal{H}}

\DeclareMathOperator{\divv}{div}
\DeclareMathOperator{\intt}{int}

\newtheorem{theorem}{Theorem}[section]
\newtheorem{proposition}[theorem]{Proposition}
\newtheorem{lemma}[theorem]{Lemma}
\newtheorem{corollary}[theorem]{Corollary}

\theoremstyle{definition}

\newtheorem{remark}[theorem]{Remark}

\newtheorem{definition}[theorem]{Definition} 

\theoremstyle{remark}

\numberwithin{equation}{section}

\definechangesauthor[name=Gianmarco, color=blue]{G}
\definechangesauthor[name=Manuel, color=purple]{M}

\begin{document}

\title[Regularity of Lipschitz boundaries]{Regularity of Lipschitz boundaries with prescribed sub-Finsler mean curvature in the Heisenberg group $\hh^1$}

\begin{abstract}
For a strictly convex set $K\subset\rr^2$ of class $C^2$ we consider its~associated sub-Finsler $K$-perimeter $|\ptl E|_K$ in $\hh^1$ and the prescribed mean curvature functional $|\ptl E|_K-\int_E f$ associated to a function $f$. Given a critical set for this functional with Euclidean Lipschitz and intrinsic regular boundary, we prove that their characteristic curves are of class $C^2$ and that this regularity is optimal. The result holds in particular when the boundary of $E$ is of class~$C^1$.
\end{abstract}

\author[G.~Giovannardi]{Gianmarco Giovannardi}
\address{Departamento de Geometría y Topología \\
Universidad de Granada \\ E--18071 Granada \\ Spain}
\email{giovannardi@ugr.es}

\author[M.~Ritoré]{Manuel Ritoré} 
\address{Departamento de Geometría y Topología \\
Universidad de Granada \\ E--18071 Granada \\ Spain}
\email{ritore@ugr.es}

\date{\today}

\thanks{Both authors have been supported by MEC-Feder
	grant MTM2017-84851-C2-1-P, Junta de Andalucía grant A-FQM-441-UGR18, MSCA GHAIA, and Research Unit MNat SOMM17/6109. The first author has also been supported by INdAM-GNAMPA project: ``Convergenze variazionali per funzionali e operatori dipendenti da campi vettoriali''}
\subjclass[2000]{53C17, 49Q10}
\keywords{Prescribed mean curvature; Heisenberg group; sub-Finsler structure; sub-Finsler perimeter; regularity of characteristic curves}

\bibliographystyle{abbrv} 

\maketitle

\thispagestyle{empty}

\section{Introduction}

The aim of this paper is to study the regularity of the characteristic curves~of the boundary of a set with continuous prescribed mean curvature in the first Heisenberg group $\hh^1$ with a sub-Finsler structure. Such a structure is defined by means of an asymmetric left-invariant norm $\norm{\cdot}_K$ in $\hh^1$  associated to a convex set $K\subset\rr^2$ containing $0$ in its interior, see \cite{2020arXiv200704683P}. We assume in this paper that $K$ has $C^2$ boundary with positive geodesic curvature. 

Following De Giorgi \cite{MR62214}, the authors of \cite{2020arXiv200704683P} defined a notion of sub-Finsler~$K$-perimeter, see also \cite{monti-finsler}. Given a measurable set $E\subset\hh^1$ and  an open subset $\Om\subset\hh^1$, it is said that $E$ has locally finite $K$-perimeter in $\Om$ if for any relatively compact open set $V\subset\Om$ we have
\[
|\ptl E|_K(V)=\sup\bigg\{\int_E\divv(U)\,d\hh^1: U\in\hhh_0^1(V), \norm{U}_{K,\infty}\le 1\bigg\}<+\infty,
\]
where $\hhh_0^1(V)$ is the space of horizontal vector fields of class $C^1$ with compact support in $V$, and $\norm{U}_{K,\infty}=\sup_{p\in V} \norm{U_p}_K$. Both the divergence and the integral are computed with respect to a fixed left-invariant Riemannian metric $g$ on $\hh^1$. When $S=\ptl E\cap\Omega$ is a Euclidean Lipschitz surface the $K$-perimeter coincides with the area functional 
\[
A_K(S)=\int_{S}\norm{N_h}_{K,*}\,d\mathcal{H}^2,
\]
where $\mathcal{H}^2$ is the $2$-dimensional Hausdorff measure associated to the left-invariant Riemannian metric $g$, $N$ is the outer unit normal to $S$, defined $\mathcal{H}^2$-a.e on $S$, $N_h$ is the horizontal projection of $N$ to the horizontal distribution in $\hh^1$ and $\norm{\cdot}_{K,*}$ is the dual norm of $\norm{\cdot}_K$. 

We say that a set $E$ with Euclidean Lipschitz boundary has prescribed $K$-mean curvature $f\in C^0(\Om)$ if, for any bounded open subset $V\subset\Om$, $E$ is a critical point of the functional
\begin{equation*}
A_K(S\cap B) -\int_{E\cap B} f\,d\hh^1.
\end{equation*}
This notion extends the classical one in Euclidean space and the one introduced in \cite{MR3412382} for the sub-Riemannian area. We refer the reader to the introduction of \cite{MR3412382} for a brief historical account and references.

We say that a set $E$ has \emph{constant} prescribed $K$-mean curvature if there exists $\lambda\in\rr$ such that $E$ has prescribed $K$-mean curvature $\lambda$. In Proposition \ref{prop:isoprescribed} we consider a set $E$ with Euclidean Lipschitz boundary and positive $K$-perimeter. We show that if $E$ is a critical point of the $K$-perimeter for variations preserving the volume up to first order then $E$ has constant prescribed $K$-mean curvature on any open set $\Om$ avoiding the singular set $S_0$ and where $|\ptl E|_K(\Om)>0$. This result can be applied to isoperimetric regions in $\hh^1$ with Euclidean Lipschitz boundary.

The main result of this paper is Theorem \ref{th:main}, where we prove that the boundary $S$ of a set $E$ with  prescribed continuous $K$-mean curvature is foliated by horizontal characteristic curves of class $C^2$ in its regular part. The minimal assumptions we require for the boundary $S$ of $E$ are to be Euclidean Lipschitz and $\mathbb{H}$-regular. The result holds in particular when the boundary of $E$ is of class $C^1$. As we point out in Remark \ref{rk:PW}, $C^2$ regularity is optimal since the Pansu-Wulff shapes obtained in \cite{2020arXiv200704683P} have prescribed constant mean curvature and their boundaries are foliated by characteristic curves with the same regularity as that of $\ptl K$,  that may be just $C^2$. In the proof of the Theorem \ref{th:main} we exploit the first variation formula of the area following the arguments developed in \cite{MR3474402, MR3412382} and make use of the biLipschitz homeomorphism considered in \cite{MR3984100}. One of the main differences in our setting is that the area functional strongly depends on the inverse $\pi_K$ of the Gauss  map of $\partial K$. Therefore  the first variation of the area depends on the derivative of the map that describes the boundary $\partial K$. In order to use the bootstrap regularity argument in \cite{MR3474402, MR3412382}  we need to invert this map on the boundary $\partial K$, that is possible since the geodesic curvature of $\partial K$ is strictly positive, see Lemma~\ref{lemmaFK}. Moreover, the $C^2$ regularity of the characteristic curves implies that, on characteristic curves of a boundary with prescribed continuous $K$-mean curvature $f$, the ordinary differential equation 
\begin{equation}
\label{eq:*}
\tag{*}
\escpr{D_Z  \, \pi_K(\nu_h) , Z}=f,
\end{equation}
is satisfied. In this equation $\nu_h=N_h/|N_h|$ is the  classical sub-Riemannian horizontal unit normal, $Z$ is the unit characteristic vector field tangent to the characteristic curves and $D$ the Levi-Civita connection associated to the left-invariant Riemannian metric $g$ on $\hh^1$ . Equation \eqref{eq:*} was proved to hold for $C^2$ surfaces in \cite{2020arXiv200704683P}. For~regularity assumptions below $\hh$-regular and Euclidean Lipschitz, equation \eqref{eq:*} holds in a suitable weak sense, a result proved in \cite{MR2223801} for the sub-Riemannian area, when $K$ coincides with the unit disk  centered at $0$.

Moreover,  in Proposition \ref{prop:subKMCE} we stress that equation \eqref{eq:*} is equivalent to
\begin{equation}
\label{eq:**}
\tag{**}
 H_D=\kappa(\pi_K(\nu_h)) f,
\end{equation}
where $H_D=\escpr{D_Z \nu_h, Z}$ is the classical sub-Riemannian mean curvature introduced in \cite{MR2223801}  and $\kappa$ is the strictly positive Euclidean  curvature of the boundary $\partial K$. A key ingredient to obtain equation \eqref{eq:**} is Lemma~\ref{lm:diff}, that exploits the ideas of  Lemma~\ref{lemmaFK} in an intrinsic setting.

This manuscript is a natural continuation of the many recent papers concerning sub-Riemannian area minimizers \cite{MR1404326, MR2354992, MR2262784, MR2165405, MR2312336, MR2472175, MR2333095,  MR2223801, MR2401420,MR3048517,MR2609016,MR2435652, MR2358000, MR3412408, MR2583494, MR2235475,MR2177813,MR2402213, 2019arXiv190505131, MR4118581,MR3794892}. The sub-Riemannian perimeter functional is a particular case of the sub-Finsler functionals considered in this paper where the convex set is the unit disk $D$ centered at $0$. In the pioneering paper \cite{MR1404326} N.~Garofalo and D.M.~Nhieu showed  the existence of sets of  minimal perimeter in Carnot-Carath\'eodory spaces satisfying the doubling property and a Poincar\'e inequality. In \cite{MR2000099} Leonardi and Rigot showed the existence of isoperimetric sets in Carnot groups. However the optimal regularity of the critical points of these variational problems involving the sub-Riemannian area is not completely understood. Indeed, even in the sub-Riemannian Heisenberg group $\hh^1$ there are several examples of non-smooth area minimizers: S. D. Pauls in \cite{MR2225631} exhibited a solution of low regularity for the Plateau problem with smooth boundary datum; on the other hand  in \cite{MR2262784,MR2448649,2008.04027} the authors provided solutions of Bernstein's problem in  $\hh^1$ that are only Euclidean Lipschitz.

In \cite{MR676380}  P. Pansu conjectured that the boundaries of isoperimetric sets in $\hh^1$ are given by the surfaces now called Pansu's spheres,  union of all sub-Riemannian geodesics of a fixed curvature joining two point in the same vertical line. This conjecture has been solved only assuming \emph{a priori} some regularity of the minimizers of the area with constant prescribed mean curvature.  In \cite{MR2435652} the authors solved the conjecture assuming that the minimizers of the area are of class $C^2$, using the description of the singular set, the characterization of area-stationary surfaces, and the ruling property of constant mean curvature surfaces developed in \cite{MR2165405}. Hence the \emph{a priori} regularity hypothesis are central  to study the sub-Riemannian isoperimetric problem. Motivated by this issue, it was shown in \cite{MR2481053} that a $C^1$ boundary of a set with continuous  prescribed mean curvature is foliated by $C^2$ characteristic curves. Regularity results for Lipschitz viscosity solutions  of the minimal surface equation were obtained in \cite{MR2583494}. Furthermore, in \cite{MR3412382} the authors generalized the previous result when the boundary $S$ is immersed in a three-dimensional contact sub-Riemannian manifold. Finally M. Galli in \cite{MR3474402} improved the result in \cite{MR3412382} only assuming that the boundary $S$ is Euclidean Lipschitz and $\hh$-regular in the sense of \cite{MR1871966}. The Bernstein problem in $\hh^1$ with Euclidean Lipschitz regularity was treated by S. Nicolussi and F. Serra-Cassano \cite{MR3984100}. Partial solutions of the sub-Riemannian isoperimetric problem have been obtained assuming Euclidean convexity \cite{MR2548252}, or symmetry properties \cite{MR2386783,MR2898770,MR2402213,MR3412408}. An analogous sub-Finsler isoperimetric problem might be considered. Candidate solutions would be the Pansu-Wulff shapes considered in \cite{2020arXiv200704683P}. See \cite{2020arXiv200704683P,monti-finsler} for partial results in the sub-Finsler isoperimetric problem and \cite{snchez2017subfinsler} for earlier work.

We have organized this paper into several sections. In Section \ref{sc:preliminaries} we introduce  sub-Finsler norms in the first Heisenberg group $\hh^1$ and their associated sub-Finsler perimeter, the notion of $\hh$-regular surfaces, intrinsic Euclidean Lipschitz graphs and  the definition of sets with prescribed mean curvature. Moreover, at the end of  this section we prove Proposition \ref{prop:isoprescribed}. Section \ref{sc:main} is dedicated to the proof of the main Theorem \ref{th:main},  that ensures that the characteristic curves are $C^2$. Finally in Section \ref{sc:KMCE} we deal with the $K$-mean curvature equation, see Proposition \ref{prop:KMCE}  and Proposition \ref{prop:subKMCE}.
\section{Preliminaries}
\label{sc:preliminaries}
\subsection{The Heisenberg group}
\label{sc:heis}
We denote by $\mathbb{H}^1$ the first Heisenberg group, defined as the $3$-dimensional Euclidean space $\rr^3$ endowed with a product $*$ defined by 
\[
(x,y,t)*(\bar{x},\bar{y} ,\bar{t})=(x+\bar{x}, y+\bar{y},t+\bar{t}+ \bar{x}y-x\bar{y}).
\]
A basis of left invariant vector fields is given by 
\[
X=\dfrac{\partial}{\partial x} + y \dfrac{\partial}{\partial t}, \qquad Y=\dfrac{\partial}{\partial y} - x \dfrac{\partial}{\partial t}, \qquad T=\dfrac{\partial}{\partial t}.
\]
For $p \in \hh^1$, the left translation by $p$ is the diffeomorphism $L_p(q) =p*q$.
The horizontal distribution $\mathcal{H}$ is the planar distribution generated by $X$ and $Y$, that coincides with the kernel of the (contact) one-form $\omega=dt-y dx+x dy$.

We shall consider on $\mathbb{H}^1$ the left invariant Riemannian metric $g= \escpr{\cdot,\cdot}$, so that $\{X, Y, T\}$ is an orthonormal basis at every point, and let  $D$ be the Levi-Civita connection associated to the Riemannian metric $g$. 
The following relations can be easily computed 
\begin{equation}
\begin{aligned}
&D_X X=0, \quad & D_Y Y=0, \qquad \, & D_T T=0\\
&D_X Y=-T, \quad & D_X T=Y, \qquad & D_Y T=-X\\
&D_Y X=T, \quad & D_T X=Y, \qquad & D_T Y=-X.\\
\end{aligned}
\end{equation}
Setting $J(U)=D_U T$ for any vector field $U$ in $\hh^1$ we get $J(X)=Y$, $J(Y)=-X$ and $J(T)=0$. Therefore $-J^2$ coincides with the identity when restricted to the horizontal distribution.
 The Riemannian volume of a set $E$ is, up to a constant, the Haar measure of the group and is denoted by $|E|$. The integral of a function $f$ with respect to the Riemannian measure by $\int f\, d\hh^1$.

\subsection{The pseudo-hermitian connection} The pseudo-hermitian connection  $\nabla$ is the only affine connection satisfying the following properties:
\begin{enumerate}
\item $\nabla$ is a metric connection,
\item $\text{Tor}(U,V)=2 \escpr{J(U),V} T$ for all vector fields $U,V$.
\end{enumerate}
In the previous line the torsion tensor $\text{Tor}(U,V) $ is given by $\nabla_U V - \nabla_V U - [U,V]$.
From the above definition and the Koszul formula  it follows easily that $\nabla X=\nabla Y=0$ and $\nabla J=0$. For a general discussion about  the pseudo-hermitian connection see for instance \cite[§ 1.2]{MR2214654}. Given a curve $\gamma: I \to \hh^1$ we denote by ${\nabla}/{ds}$ the covariant derivatives induced by the pseudo-hermitian connection along $\gamma$.

\subsection{Sub-Finsler norms}
Given a convex set $K\subset\rr^2$ with $0\in\intt(K)$ an~ associated asymmetric norm $\norm{\cdot}$ in $\rr^2$, we define on $\hh^1$ a left-invariant norm $\norm{\cdot}_K$ on the horizontal distribution by means of the equality
\[
(\norm{fX+gY}_K)(p)=\norm{(f(p),g(p))},
\] 
for any $p\in\hh^1$. The dual norm is denoted by $\norm{\cdot}_{K,*}$. 

If  the boundary of $K$ is of class $C^\ell$, $\ell\ge 2$, and the geodesic curvature of $\ptl K$ is strictly positive, we say that $K$ is of class $C^\ell_+$. When $K$ is of class $C^2_+$, the outer Gauss map $N_K$ is a diffeomorphism from $\ptl K$ to $\sph^1$ and the map
\[
\pi_K(fX+gY)=N_K^{-1}\bigg(\frac{(f,g)}{\sqrt{f^2+g^2}}\bigg),
\]
defined for non-vanishing horizontal vector fields $U=fX+gY$, satisfies
\[
\norm{U}_{K,*}=\escpr{U,\pi_K(U)}.
\]
See \S~2.3 in \cite{2020arXiv200704683P}.

\subsection{Sub-Finsler perimeter}
Here we summarize some of the results contained in subsection 2.4 in \cite{2020arXiv200704683P}.

Given a convex set $K\subset\rr^2$ with $0\in\intt(K)$, the norm $\norm{\cdot}_K$ defines a perimeter functional: given a measurable set $E\subset\hh^1$ and  an open subset $\Om\subset\hh^1$, we say that $E$ has locally finite $K$-perimeter in $\Om$ if for any relatively compact open set $V\subset\Om$ we have
\[
|\ptl E|_K(V)=\sup\bigg\{\int_E\divv(U)\,d\hh^1: U\in\hhh_0^1(V), \norm{U}_{K,\infty}\le 1\bigg\}<+\infty,
\]
where $\hhh_0^1(V)$ is the space of horizontal vector fields of class $C^1$ with compact support in $V$, and $\norm{U}_{K,\infty}=\sup_{p\in V} \norm{U_p}_K$. The integral is computed with respect to the Riemannian measure $d\hh^1$ of the left-invariant Riemannian metric $g$. When $K=D$, the closed unit disk centered at the origin of $\rr^2$, the $K$-perimeter coincides with classical sub-Riemannian perimeter.

If $K,K'$ are bounded convex bodies containing $0$ in its interior, there exist~constants $\alpha,\beta>0$ such that
\[
\alpha \norm{x}_{K'}\le \norm{x}_K\le \beta \norm{x}_{K'},\quad \text{for all }x\in\rr^2,
\]
and it is not difficult to prove that
\begin{equation*}
\tfrac{1}{\beta} |\ptl E|_{K'}(V)\leq |\ptl E_K|(V)\leq \tfrac{1}{\alpha} |\ptl E|_{K'}(V).
\end{equation*}
As a consequence,  $E$ has locally finite $K$-perimeter if and only if it has locally finite $K'$-perimeter. In particular, any set with locally finite $K$-perimeter has locally finite sub-Riemannian perimeter.

Riesz Representation Theorem implies the existence of a $|\ptl E|_K$ measurable vector field $\nu_K$ so that for any horizontal vector field $U$ with compact support of class $C^1$ we have
\[
\int_\Om\divv(U)\,d\hh^1=\int_\Om\escpr{U,\nu_K}\,d|\ptl E|_K.
\]
In addition, $\nu_K$ satisfies $|\ptl E|_K$-a.e. the equality $\norm{\nu_K}_{K,*}=1$, where $\norm{\cdot}_{K,*}$ is the dual norm of $\norm{\cdot}_K$.  

Given two convex sets $K,K'\subset \rr^2$ containing $0$ in their interiors, we have the following representation formula for the sub-Finsler perimeter measure $|\partial E|_K$ and the vector field $\nu_K$
\begin{equation*}
|\partial E|_K=\norm{\nu_{K'}}_{K,*}|\partial E|_{K'},\quad \nu_{K}=\frac{\nu_{K'}}{\norm{\nu_{K'}}_{K,*}}.
\end{equation*}
Indeed, for the closed unit disk $D\subset\rr^2$ centered at $0$ we know that in the Euclidean Lipschitz case $\nu_D=\nu_h$ and $|N_h|=\norm{N_h}_{D,*}$ where $N$ is the \emph{outer} unit normal. Hence we have
\begin{equation*}
|\partial E|_K=\norm{\nu_h}_{K,*}d|\partial E|_D, \quad \nu_K=\frac{\nu_h}{\norm{\nu_h}_{K,*}}.
\end{equation*}
Here $|\partial E|_D$ is the standard sub-Riemannian measure. Moreover, $\nu_h=N_h/|N_h|$ and $|N_h|^{-1}d|\ptl E|_D=dS$, where $dS$ is the standard Riemannian measure on $S$. Hence we get, for a set $E$ with Euclidean Lipschitz boundary $S$
\begin{equation}
\label{eq:AKlipschitz}
|\ptl E|_K(\Om)=\int_{S\cap\Om}\norm{N_h}_{K,*}\,dS,
\end{equation}
where $dS$ is the Riemannian measure on $S$, obtained from the area formula using a local Lipschitz parameterization of $S$, see Proposition~2.14 in \cite{MR1871966}. It coincides with the $2$-dimensional Hausdorff measure associated to the Riemannian distance induced by $g$. We stress that here $N$ is the \emph{outer} unit normal. This choice is important because of the lack of symmetry of $\norm{\cdot}_K$ and $\norm{\cdot}_{K,*}$.

\subsection{Immersed surfaces in $\hh^1$}
\label{sc:surfaceinH}

Following \cite{MR2223801,MR1871966} we provide the following definition.
\begin{definition}[$\mathbb{H}$-regular surfaces]
A real measurable function $f$ defined on an open set $\Omega\subset \mathbb{H}^1$  is of class $C_{\mathbb{H}}^1(\Omega)$ if the distributional derivative $\nabla_{\mathbb{H}} f= (X f, Y f)$ is represented by a continuous function.

We say that $S \subset \mathbb{H}^1$ is an $\mathbb{H}$-regular surface if for each $p \in \mathbb{H}^1$ there exist a neighborhood  $U$ and a function $f \in C_{\mathbb{H}}^1(U)$ such that $\nabla_{\mathbb{H}} f \ne 0$ and $S \cap U=\{f=0\}$.
Then the continuous horizontal unit normal  is given by 
\[
\nu_h=\dfrac{\nabla_{\mathbb{H}} f}{|\nabla_{\mathbb{H}} f|}.
\]
\end{definition}

Given an oriented Euclidean Lipschitz surface $S$ immersed in $\hh^1$, its unit normal~$N$ is defined $\mathcal{H}^2$-a.e. in $S$, where $\mathcal{H}^2$ is the $2$-dimensional Hausdorff measure associated to the Riemannian distance induced by $g$.  In case $S$ is the boundary of a set $E \subset \hh^1 $, we always choose the outer unit normal. We say that a point $p$ belongs to the \textit{singular set} $S_0$  of $S$ if  $p \in S$ is a differentiable point  and the tangent space $T_pS$ coincides with the horizontal distribution $\mh_p$. Therefore the  horizontal projection of the normal $N_h$ at singular points vanishes. In $S \smallsetminus S_0$ the horizontal unit normal $\nu_h$ is defined $\mathcal{H}^2$-a.e. by
\[
\nu_h=\frac{N_h}{|N_h|},
\]
where $N_h$ is the horizontal projection of the normal $N$. The vector field $Z$ is defined $ \mathcal{H}^2$-a.e. on $S\smallsetminus S_0'$ by $Z=J(\nu_h)$, and it is tangent to $S$ and horizontal.

 $\hh$-regularity plays an important role in the regularity theory of sets of finite sub-Riemannian perimeter. In \cite{MR1871966}, B. Franchi, R. Serapioni and F. Serra-Cassano proved that the boundary of such a set is composed of $\hh$-regular surfaces and a singular set of small measure.

\subsection{Sets with prescribed mean curvature}
Consider an open set $\Om\subset M$, and an integrable function $f\in L^1_{loc}(\Om)$. We  say that a set of locally finite $K$-perimeter $E\subset\Om$ has \emph{prescribed $K$-mean curvature $f$ in $\Om$} if, for any bounded open set $B\subset\Om$, $E$ is a critical point of the functional
\begin{equation}
\label{eq:a-hv}
|\ptl E|_K(B)-\int_{E\cap B} f\,d\hh^1.
\end{equation}

If $S=\ptl E\cap\Omega$ is a Euclidean Lipschitz surface then $S$ has prescribed $K$-mean curvature $f$ if it is a critical point of the functional
\begin{equation}\label{eq:prescribedfunctional}
A_K(S\cap B) -\int_{E\cap B} f\,d\hh^1,
\end{equation}
for any bounded open set $B\subset\Om$.

If $E$ has boundary $S=\ptl E\cap\Om$ of class $C^2$, standard arguments imply that $E$ has prescribed $K$-mean curvature $f$ in $\Om$ if and only if $H_K=f$, where $H_K$ is the $K$-mean curvature
\[
H_K=\escpr{D_Z\pi_K(\nu_h),Z},
\]
and $\nu_h$ is the \emph{outer} horizontal unit normal, see \cite{2020arXiv200704683P}. Since by \cite [Lemma 2.1]{2020arXiv200704683P} the Levi-Civita connection $D$ and the pseudo-hermitian connection $\nabla$ coincide for horizontal vector fields, we obtain that 
\[
H_K=\escpr{D_Z\pi_K(\nu_h),Z}= \escpr{\nabla_Z\pi_K(\nu_h),Z}.
\]
It is important to remark that the mean curvature $H_K$ strongly depends on the choice of $\nu_h$. When $K$ is centrally symmetric, $\pi_K(-u)=-\pi_K(u)$ and so the mean curvature changes its sign when we take $-\nu_h$ instead of $\nu_h$. When $K$ is not centrally symmetric, there is no relation between the mean curvatures associated to $\nu_h$ and $-\nu_h$.

A set $E\subset\hh^1$ with Euclidean Lischiptz boundary has locally finite $K$-perimeter: we know that it has locally bounded sub-Riemannian perimeter by Proposition~2.14 in \cite{MR1871966} and we can apply the perimeter estimates in \S~2.3. Letting $\mathcal{H}^2$ be the Riemannian $2$-dimensional Hausdorff measure, the Riemannian  outer unit normal $N$ is defined $\mathcal{H}^2$-a.e. in $\ptl E$, and it can be proven that
\begin{equation}
\label{eq:lipsintform}
|\ptl E|_K(V)=\int_{\ptl E\cap V}\norm{N_h}_{K,*}\,d\mathcal{H}^2.
\end{equation}

We say that a set $E$ of locally finite $K$-perimeter in an open set $\Om\subset\hh^1$ has \emph{constant} prescribed $K$-mean curvature if there exists $\lambda\in\rr$ such that $E$ has prescribed $K$-mean curvature $\lambda$. This means that $E$ is a critical point of the functional $E\mapsto |\ptl E|_K(B)-\lambda|E\cap B|$ for any bounded open set $B\subset\Om$. 

Our next result implies that Euclidean Lipschitz isoperimetric boundaries (for the $K$-perimeter) have constant prescribed $K$-mean curvature.

\begin{proposition}
\label{prop:isoprescribed}
Let $E\subset\hh^1$ be a bounded set with Euclidean Lipschitz boundary. Assume that $E$ a critical point of the $K$-perimeter for variations preserving the volume of $E$ up to first order. Let $\Om\subset\hh^1$ be an open set so that $\Om\cap S_0=\emptyset$ and $|\ptl E|_K(\Om)>0$. Then $E$ has constant prescribed $K$-mean curvature in $\Om$.
\end{proposition}

\begin{proof}
Since the $K-$perimeter of $E$ in $\Om$ is positive there exists a horizontal vector field $U_0$ with compact support in $\Om$ so that $\int_{E}\divv U_0\,d\hh^1>0$.  Let $\{\psi_s\}_{s\in\rr}$ be the flow associated to $U_0$ and define
\begin{equation}
\label{eq:defH0}
H_0=\frac{\tfrac{d}{ds}|_{s=0}\, A_K(\psi_s(S))}{\tfrac{d}{ds}|_{s=0}|\psi_s(E)|}.
\end{equation}

Let $W$ any vector field with compact support in $\Om$ and associated flow $\{\varphi_s\}_{s\in\rr}$. Choose $\la\in\rr$ so that $W-\la U_0$ satisfies
\[
\frac{d}{ds}\bigg|_{s=0}|\varphi_s(E)|-\la\,\frac{d}{ds}\bigg|_{s=0}|\psi_s(E)|=0.
\]
This means that the flow of $W-\la U_0$ preserves the volume of $E$ up to first order. 

By our assumption on $E$ we get
\[
Q(W-\la U_0)=0,
\]
where $Q$ is defined in \eqref{eq:defQ}. Now Lemma~\ref{lem:1stderlipschitz} implies $Q(W)=\la Q(U_0)$ and, from the definition of $H_0$, we get
\[
Q(W)=\la Q(U_0)=\la H_0 \frac{d}{ds}\bigg|_{s=0}|\psi_s(E)|=H_0\frac{d}{ds}\bigg|_{s=0} |\varphi_s(E)|.
\]
This implies that $E$ is a critical point of the functional $E\mapsto |\ptl E|_K-H_0|E|$ and so it has prescribed $K$-mean curvature equal to the constant $H_0$.
\end{proof}

\begin{lemma}
\label{lem:1stderlipschitz}
Let $E\subset\hh^1$ be a bounded set with Euclidean Lipschitz boundary $S$. Let $\Om\subset\hh^1$ be an open set such that $\Om\cap S_0=\emptyset$. Let $U$ be a vector field with compact support $\Om$ and $\{\varphi_s\}_{s\in\rr}$ the associated flow. Then the derivative
\begin{equation}
\label{eq:defQ}
Q(U)=\frac{d}{ds}\bigg|_{s=0} A_K(\varphi_s(S))
\end{equation}
exists and is a linear function of $U$.
\end{lemma}

\begin{proof}
For every $s\in\rr$, the set $\varphi_s(E)$ has Euclidean Lipschitz boundary and so it has finite $K$-perimeter. By Rademacher's Theorem, the set
\[
B=\{p\in S: S \text{ is not differentiable at }p\}
\]
has $\mathcal{H}^2$-measure equal to $0$.

For any $p\in S\setminus B$ we take the curve $\sg(s)=\varphi_s(p)$.  For every $s\in\rr$ the surface $\varphi_s(S)$ is differentiable at $\sg(p)$ and the vector field $W(s)=((N_s)_h)_{\sg(s)}$, where $N_s$ is the outer unit normal to $\varphi_s(\ptl E)$, is differentiable along the curve $\sg$. Let us estimate the quotient
\begin{equation}
\label{eq:incremental}
\frac{\norm{W(s+h)}_{K,*}-\norm{W(s)}_{K_*}}{h}.
\end{equation}
Writing $W(s)=f(s)X_{\sg(s)}+g(s)Y_{\sg(s)}$ we have $\norm{W(s)}_{K,*}=\norm{(f(s),g(s))}$, where $\norm{\cdot}$ is the planar asymmetric norm associated to the convex set $K$. We have 
\begin{align*}
\big|\norm{W(s+h)}_{K,*}-\norm{W(s)}_{K_*}\big|&\le\norm{(f(s+h)-f(s),g(s+h)-g(s))}
\\
&\le C\,\big(|f(s+h)-f(s)|+|g(s+h)-g(s)|\big),
\end{align*}
for a constant $C>0$ that only depends on $K$. The derivates of $f$ and $g$ can be estimated in terms of the covariant derivative $\tfrac{D}{ds}W=\tfrac{D}{ds}(N_s)_h$ along $\sg$. Since
\[
\bigg|\frac{D}{ds} (N_s)_h\bigg|\le \big|\divv_{\varphi_s(S)}(U)\big|
\]
we get an uniform estimate on the derivatives of $f$ and $g$ independent of $p$. So the quotient \eqref{eq:incremental} is uniformly bounded above by a constant independent of $p$.

To compute the derivative of $A_K(\varphi_s(S))$ at $s=0$ we write
\[
A_K(\varphi_s(S))=\int_{S}\big(\norm{(N_s)_h}_{K,*}\circ\varphi_s\big)\,\text{Jac}(\varphi_s)\,d\mathcal{H}^2
\]
The uniform estimate of the quotient \eqref{eq:incremental} allows us to apply Lebesgue's dominated convergence theorem and Leibniz's rule to compute the derivative of $A_K(\varphi_s(S))$, given by
\[
\int_S\frac{d}{ds}\bigg|_{s=0}\bigg(\big(\norm{(N_s)_h}\circ\varphi_s\big)\,\text{Jac}(\varphi_s)\bigg)\,d\mathcal{H}^2.
\]

Given a point $p\in (S\smallsetminus B)\cap\text{supp}(U)$,  since $\text{supp}(U)\subset\Om$ and $\Om\cap S_0=\emptyset$ we get $(N_h)_p\neq 0$ and so
\begin{align*}
\frac{D}{ds}\bigg|_{s=0} \norm{(N_s)_h}_{K,*}&(\sg(s))=\frac{D}{ds}\bigg|_{s=0} \escpr{(N_s)_h,\pi_K((N_s)_h)}(\sg(s)) \\
&=\escpr{\frac{D}{ds}\bigg|_{s=0} (N_s)_h,(N_h)_p}+\escpr{(N_h)_p,(d\pi_K)\bigg(\frac{D}{ds}\bigg|_{s=0} (N_s)_h\bigg)}.
\end{align*}
Since
\[
\frac{D}{ds}\bigg|_{s=0} (N_s)_h=\frac{D}{ds}\bigg|_{s=0} N-\escpr{\frac{D}{ds}\bigg|_{s=0} N,T}\,T,
\]
and
\[
\frac{D}{ds}\bigg|_{s=0} N=\sum_{i=1}^2 \escpr{N_p,\nabla_{e_i}U}\,e_i,
\]
where $e_i$ is an orthonormal basis of $T_p(\ptl E)$, we get that
\[
\frac{D}{ds}\bigg|_{s=0} \norm{N_s}_{K,*}
\]
is a linear function $L(U)$ of $U$.
\end{proof}

\begin{remark}
Proposition~\ref{prop:isoprescribed} can be applied to isoperimetric regions in $\hh^1$ with Euclidean Lipschitz boundary. Of course, the regularity of isoperimetric regions in $\hh^1$ is still an open problem.
\end{remark} 

\subsection{Intrinsic Euclidean Lipschitz graphs on a vertical plane in $\hh^1$}
\label{sc:intgraph}
We denote by $\text{Gr}(u)$ the \textit{intrinsic} graph (Riemannian normal graph) of the Lipschitz function $u: D \to \rr,$ where $D$ is a domain in a vertical plane. Using Euclidean rotations about the vertical axis $x=y=0$, that are isometries of the Riemannian metric $g$, we may assume that $D$ is contained in the plane $y=0$. Since the~vector field $Y$ is a unit normal to this plane, the intrinsic graph $\text{Gr}(u)$ is given by $\{ \exp_p (u(p) Y_p)  : p \in D\}$, where $\exp$ is the exponential map of $g$, and can be parameterized by the map
\[
\Phi^u (x,t)=(x, u(x,t), t- x u(x,t)).
\]

The tangent plane to any point in $S=\text{Gr}(u)$ is generated by the vectors
\begin{align*}
\Phi_x^u&=(1, u_x,-u-x u_x)= X+ u_x Y-2 u T,\\
\Phi_t^u&=(0,u_t,1-x u_t)= u_t Y+T
\end{align*}
and the characteristic direction is given by $Z= {\tilde{Z} }/{|\tilde{Z}|}$ where 
\begin{equation}
\label{eq:Z}
\tilde{Z}= X+ ( u_x+ 2 u u_t )Y.
\end{equation}
A unit normal to $S$ is given by $N= \tilde{N} / |\tilde{N}|$ where 
\[
\tilde{N}=  \Phi_x^u \times \Phi_t^u= (u_x + 2 u u_t) X-Y+u_t T
\]
and $\text{Jac}(\Phi^u)=|\Phi_x^u \times \Phi_t^u|=|\tilde{N}|$. 
Therefore the horizontal projection of the unit normal to $S$ is given by  $N_h=\tilde{N}_h / |\tilde{N}|$, where
$
\tilde{N}_h=(u_x + 2 u u_t) X-Y.
$
Observe that $J(Z)=-\nu_h$.

We also assume that $S=\text{Gr}(u)$ is an $\mathbb{H}$-regular surface, meaning that $\tilde{N}_h$ and  $\tilde{Z}$ in \eqref{eq:Z} and   are continuous. 
 Hence also  $(u_x+2u u_t)$ is continuous.
\begin{remark}
\label{rk:horcur}
Let $\ga(s)=(x,t)(s)$ be a $C^1$ curve in $D$ then 
\[
\Gamma(s)=(x,u(x,t), t-x u(x,t))(s) \subset \text{Gr}(u)
\]
is also $C^1$ and 
\[
\Gamma'(s)=x' X +(x'u_x+t'u_t)Y+(t'-2ux')T.
\]
In particular horizontal curves in $\text{Gr}(u)$  satisfy the ordinary differential equation
\begin{equation}
\label{eq:horode}
t'=2 u(x,t) x'.
\end{equation}
\end{remark}

\medskip

From \eqref{eq:AKlipschitz}, the sub-Finsler $K$-area for a Euclidean Lipschitz surface $S$ is
\[
A_K(S)= \int_{S} \| N_h \|_{K,*} dS,
\]
where $\| N_h \|_{K,*}= \escpr{N_h, \pi(N_h)}$ with $\pi=(\pi_1,\pi_2)=\pi_K$ and $dS$ is the Riemannian area measure. Therefore when we consider the intrinsic graph $S= \text{Gr}(u)$ we obtain 
\begin{align*}
A(\text{Gr}(u))&= \int_D \escpr{\tilde{N}_h, \pi(\tilde{N}_h)} \ dx dt\\
&= \int_D  (u_x + 2 u u_t)   \pi_1(u_x + 2 u u_t, -1) - \pi_2(u_x + 2 u u_t, -1)  \, dx dt.
\end{align*}
Observe that the $K$-perimeter of a set was defined in terms of the \emph{outer} unit normal. Hence we are assuming that $S$ is the boundary of the \emph{epigraph} of $u$.

Given $v \in C_0^{\infty}(D)$, a straightforward computation shows that 
\begin{equation}
\label{eq:fv0}
\dfrac{d}{ds}\Big|_{s=0} A(\text{Gr}(u+sv))= \int_D (v_x+ 2v u_t+ 2 u v_t) M dx dt,
\end{equation}
where 
\begin{equation}
\label{eq:M}
M=F(u_x+2uu_t),
\end{equation}
 and $F$ is the function
\begin{equation}
\label{def:F}
F(x)= \pi_1(x,-1) + x \dfrac{\ptl \pi_1}{\partial x} (x,-1)- \dfrac{\partial \pi_2}{\partial x} (x,-1).
\end{equation}
Since $(u_x+2u u_t)$ is continuous  and $\pi$ is at least $C^{1}$ the function $M$ is continuous.

\section{Characteristic curves are $C^2$}
\label{sc:main}
Here we prove our main result, that characteristic curves in an intrinsic Euclidean Lipschitz $\hh$-regular surface with continuous prescribed $K$-mean curvature are of class $C^2$. The reader is referred to Theorem 4.1 in \cite{MR3412382} for a proof of the the sub-Riemannian case. The proof of Theorem~\ref{th:main} depends on Lemmas~\ref{lemmaFK} and \ref{lm:MC1}.

\begin{theorem}
\label{th:main}
Let $K$ be a $C_+^2$ convex set in $\rr^2$ with $0\in\intt(K)$ and $\norm{\cdot}_K$ the associated left-invariant norm in $\hh^1$. Let $\Omega \subset \mathbb{H}^1$ be an open set and $E \subset \Omega$ a set of prescribed $K$- mean curvature $f \in C^0(\Omega)$ with an Euclidean Lipschitz and  $\mathbb{H}$-regular boundary $S$. Then the characteristic curves of $S\cap\Om$ are of class $C^2$.
\end{theorem}

\begin{proof}
By the Implicit Function Theorem for $\hh$-regular surfaces,  see Theorem~6.5 in \cite{MR1871966}, given a point $p \in S$, after a rotation about the vertical axis, there exists an open neighborhood $B \subset \mathbb{H}^1 $ of  $p$ such that $B \cap S$ is the intrinsic graph $\text{Gr}(u)$ of a function $u: D \to \rr$, where $D$ is a domain in the vertical plane $y=0$, and $B\cap E$ is the epigraph of $u$. The function $u$ is Euclidean Lipschitz by our assumption. .Since $\text{Gr}(u)$ has prescribed continuous mean curvature $f$, from equation \eqref{eq:fv0} we get
\begin{equation}
\label{eq:fv}
 \int_D (v_x+ 2v u_t+ 2 u v_t) M + f v  \, dx dt=0,
\end{equation}
for each $v \in C_0^{\infty}(D)$. The function $M$ is defined in \eqref{eq:M}. By Remark~4.3 in \cite{MR3412382} implies that \eqref{eq:fv} holds for each $v \in C_0^{0}(D)$ for which $v_x + 2uv_t$ exists and is continuous.

Let $\Gamma(s)$ be a characteristic horizontal curve passing through $p$ whose velocity is the vector field $\tilde{Z}$ defined in \eqref{eq:Z}, that only depends  on $u_x+ 2 u u_t$. Since $S$ is $\mathbb{H}$-regular the function $u_x+ 2 u u_t$ is continuous and $\Gamma(s)$ is of class $C^1$. Let us consider the function $F$ defined in \eqref{def:F}
and define
\[
g(s)= (u_x+2u u_t)_{\Ga(s)}.
\]
Hence $F(g(s))=M(s)$. The function $F$ is $C^1$ for any convex set $K$ of class $C^2_+$ and, from Lemma \ref{lemmaFK}, we obtain that $F'(x)>0$ for each $x \in \rr$. Therefore $F^{-1}$ is also $C^1$ and $g(s)=F^{-1}(M(s))$. Thanks to Lemma \ref{lm:MC1} we obtain that $M$ is $C^1$ along~$\Gamma$ and we conclude that also $g$ is $C^1$ along $\Gamma$. So $\tilde{Z}$ is $C^1$ and the curve $\Gamma$ is~$C^2$.
\end{proof}

\begin{lemma}
\label{lemmaFK}
Let $K\subset \rr^2$  be a convex body of class $C_+^2$ such that $0 \in\intt(K)$. Then  the function $F$ defined in \eqref{def:F} is $C^1$ and  $F'(x)>0$ for each $x \in \rr$.
\end{lemma}

\begin{proof}
Parameterize the lower part of the boundary of the convex body $K$ by a function $\phi$  defined on a closed interval $I \subset \rr$. The function $\phi$ is of class $C^2$ in $\mathring{I}$ and the graph becomes vertical at the endpoints of $I$. As $K$ is of class $C_+^2$ we have $\phi''(x)>0$ for each $x \in \rr$. Take $x\in \rr$, then we have
\[
\pi(x,-1)=N_K^{-1}\bigg(\dfrac{(x,-1)}{\sqrt{1+x^2}}\bigg),
\]
where $N_K$ is the outer unit normal to $\partial K$. Let $\varphi(x) \in \mathring{I} $ be the point where 
\[
(\varphi(x), \phi (\varphi(x)))=\pi(x,-1).
\]
Therefore, if we consider the normal $N_K$ of the previous equality we obtain 
\[
\dfrac{(\phi'(\varphi(x)),-1)}{\sqrt{1+(\phi'(\varphi(x)))^2}}=\dfrac{(x,-1)}{\sqrt{1+x^2}}.
\]
Hence $\phi'(\varphi(x))=x$ and so $\varphi$ is the inverse of $\phi'$, that is invertible since $\phi''(x)>0$ for each $x \in \rr$. Notice that 
\begin{align*}
F(x)&=\pi_1(x,-1) + x \dfrac{\partial \pi_1}{\partial x} (x,-1)- \dfrac{\partial \pi_2}{\partial x} (x,-1)\\
&=\varphi(x)+x\varphi'(x)-\phi'(\varphi(x))\varphi'(x)=\varphi(x),
\end{align*}
since $\phi'(\varphi(x))=x$. Hence we obtain  
\[
F'(x)=\varphi'(x)=\dfrac{1}{\phi''(\varphi(x))} >0
\] 
for each $x \in \rr$.
\end{proof}

\begin{lemma}
\label{lm:MC1}
Let $\Omega \subset \mathbb{H}^1$ be an open set and $E \subset \Omega$ a set of prescribed $K$-mean curvature $f \in C^0(\Omega)$ with Euclidean Lipschitz and  $\mathbb{H}$-regular boundary $S$. Then the function $M$ defined in \eqref{eq:M} is of class $C^1$ along  characteristic curves. Moreover, the differential equation
\[
\frac{d}{ds} M(\ga(s))=f(\ga(s))
\]
is satisfied along any characteristic~curve $\ga$.
\end{lemma}
\begin{proof}
Let $\Gamma(s)$ be a characteristic curve passing through $p$  in $\text{Gr}(u)$. Let $\ga(s)$ be the projection of $\Gamma(s)$ onto the $xt$-plane, and $(a,b) \in D$ the projection of $p$ to the $xt$-plane. We  parameterize $\ga$ by  $s\to (s,t(s))$. By Remark~\ref{rk:horcur} the curve $s \to (s,t(s))$  satisfies the ordinary differential equation $t'=2u$. For $\eps$ small enough, Picard-Lindel\"of's theorem implies the existence of $r>0$ and  a solution $t_{\eps}: ]a-r,a+r[\to \rr$ of the Cauchy problem 
\begin{equation}
\label{eq:teps}
\begin{cases}
t_{\eps}'(s)=2u(s, t_{\eps}(s)),\\
t_{\eps}(a)=b+\eps.
\end{cases}
\end{equation}
We define $\ga_{\eps}(s)=(s,t_{\eps}(s))$  so that $\ga_0=\ga$. Here we exploit an argument similar to the one developed in \cite{MR3984100}.
By Theorem~2.8 in \cite{MR2961944} we gain that $t_{\eps}$ is Lipschitz with respect to $\eps$ with Lipschitz constant less than  or equal to $e^{Lr}$.  Fix $s \in ]a-r,a+r[$, the inverse of the function $\eps \to t_{\eps}(s)$ is given by $\bar{\chi}_{t}(-s)= \chi_{t}(-s)-b$ where $\chi_t$ is the unique solution of the following Cauchy problem
\begin{equation}
\label{eq:tau}
\begin{cases}
\chi_{t}'(\tau)=2u(\tau, \chi_{t}(\tau))\\
\chi_{t}(a+s)=t.
\end{cases}
\end{equation}
Again by Theorem~2.8 in \cite{MR2961944} we have that $\bar{\chi}_{t}$ is Lipschitz continuous with respect to $t$, thus the function $\eps \to t_{\eps}$ is a locally biLipschitz homeomorphisms.

We  consider the following Lipschitz coordinates 
\begin{equation}
\label{eq:Gcc}
G(\xi,\eps)=(\xi,t_{\eps}(\xi))=(s,t)
\end{equation}
around the characteristic curve passing through  $(a,b)$. Notice that, by the uniqueness result for  \eqref{eq:teps},  $G$ is injective. 
Given $(s,t)$ in the image of $G$ using the inverse function $\bar{\chi}_{t}$ defined in \eqref{eq:tau} we find $\eps$ such that $t_{\eps}(s)=t$, therefore $G$ is surjective. By the Invariance of Domain Theorem \cite{MR1511684},   is a homeomorphism.
The Jacobian of $G$ is defined by  
\begin{equation}
\label{eq:JacG}
\textbf{J}_G= \det \left(\begin{array}{cc} 1 & 0\\ t_{\eps}' & \frac{\partial t_{\eps}}{\partial \eps} \end{array}\right)=\frac{\partial t_{\eps}}{\partial \eps}(s)
\end{equation}
almost everywhere in $\eps$. Any function $\varphi$ defined on $D$ can  be considered as a function of the variables $(\xi,\eps)$ by making $\tilde{\varphi}(\xi,\eps)=\varphi(\xi,t_{\eps}(\xi))$. Since the function $G$ is $C^1$ with respect to $\xi$ we have 
\[
 \frac{\partial \tilde{\varphi}}{\partial \xi}= \varphi_x+  t_{\eps}' \, \varphi_t =\varphi_x+  2 u  \varphi_t.   
 \]
Furthermore, by \cite[Theorem 2 in Section 3.3.3]{Evans2015} or \cite[Theorem 3]{MR1201446}, we may apply the change of variables formula for Lipschitz maps. Assuming that the support of $v$  is contained in a sufficiently small neighborhood of $(a,b)$, we can express the integral \eqref{eq:fv} as
\begin{equation}
\label{eq:fv1}
 \int_I \big(\int_{a-r}^{a+r}( (\frac{\partial \tilde{v}}{\partial \xi}+ 2 \tilde{v} \, \tilde{u}_t) \tilde{M} + \tilde{f} \tilde{v} )\frac{\partial t_{\eps}}{\partial \eps} \,  d\xi \big) d\eps=0,
\end{equation}
where $I$ is a small interval containing $0$. Instead of $\tilde{v}$ in \eqref{eq:fv1} we consider the function $\tilde{v} h/(t_{\eps+h}-t_{\eps})$, where $h$ is a small enough parameter. Then we obtain 
\begin{align*}
\dfrac{\partial }{\partial \xi} \left({  \dfrac{\tilde{v} h }{(t_{\eps+h}-t_{\eps})}} \right)&= \frac{\partial \tilde{v}}{\partial \xi} \dfrac{ h }{(t_{\eps+h}-t_{\eps})}-  \tilde{v} h \dfrac{t_{\eps+h}'-t_{\eps}'} {(t_{\eps+h}-t_{\eps})^2}\\
&=\frac{\partial \tilde{v}}{\partial \xi} \dfrac{ h }{(t_{\eps+h}-t_{\eps})}-  2 \tilde{v} h \dfrac{u(\xi, t_{\eps+h}(\xi))-u(\xi,t_{\eps}(\xi)} {(t_{\eps+h}-t_{\eps})^2},
\end{align*}
that tends to
\[
\left(\frac{\partial t_{\eps}}{\partial \eps}\right)^{-1} \left( \frac{\partial \tilde{v}}{\partial \xi} -2  \tilde{v} \tilde{u}_t \right) \qquad a.e. \ \text{in} \ \eps,
\]
when $h$ goes to $0$. Putting $\tilde{v} h/(t_{\eps+h}-t_{\eps})$ in \eqref{eq:fv1} instead of $\tilde{v}$ we gain
\[
\int_I \left(\int_{a-r}^{a+r}  \dfrac{h \frac{\partial t_{\eps}}{\partial \eps}}{(t_{\eps+h}-t_{\eps})}\left(\frac{\partial \tilde{v}}{\partial \xi}+ 2 \tilde{v} \, (\tilde{u}_t- \dfrac{\tilde{u}(\xi, {\eps+h})-\tilde{u}(\xi,{\eps})} {(t_{\eps+h}-t_{\eps})}) \right) \tilde{M} + \tilde{f} \tilde{v}  \, d\xi \right) d\eps=0.
\]
Using Lebesgue's dominated convergence theorem and letting $h\to 0$ we have
\begin{equation}
\label{eq:tildefvf}
\int_I \left(\int_{a-r}^{a+r}  \frac{\partial \tilde{v}}{\partial \xi} \tilde{M} +\tilde{f} \tilde{v}  \, d\xi \right) d\eps=0.
\end{equation}
Let $\eta: \rr \to \rr$ be a positive function compactly supported in $I$ and for $\rho>0$ we consider the family $\eta_{\rho}(x)=\rho^{-1} \eta(x/\rho)$, that weakly converge to the Dirac delta distribution. Putting  the test functions $\eta_{\rho}(\eps) \psi(\xi)$ in  \eqref{eq:tildefvf} and letting $\rho\to 0$ we get
\begin{equation}
\label{eq:fvacc}
\int_{a-r}^{a+r}  \psi'(\xi) \tilde{M}(\xi,0) +\tilde{f}(\xi,0) \psi(\xi)  \, d\xi=0,
\end{equation}
for each $\psi \in C^{\infty}_0((a-r,a+r))$. Since $u_x+2u u_t $ is continuous,  $M$ in  \eqref{eq:M} is  continuous, thus also $\tilde{M}$. Hence thanks to Lemma \ref{lm:intbypart} we conclude that $M$ is $C^1$ along $\gamma$, thus by Remark \ref{rk:horcur} is also $C^1$ along $\Gamma$.

Since $M$ is $C^1$ along the characteristic curve, we can integrate by parts in equation \eqref{eq:fvacc}  to obtain
\[
\int_{a-r}^{a+r}  \left(-  \tilde{M}'(0,\xi)  + \tilde{f}(0,\xi)  \right) \psi(\xi)  \, d\xi=0,
\]
for each $\psi \in C^{\infty}_0((a-r,a+r))$. That means that $M$ satisfies the equation 
\[
\dfrac{d}{ds} M(\ga(s))=f(\ga(s))
\]
along characteristic curves.
\end{proof}

\begin{lemma}[{\cite[Lemma~4.2]{MR3412382}}]
\label{lm:intbypart}
Let $J \subset \rr$ be an open interval and $g,h \in C^{0}(J)$. Let $H\in C^1(J)$ be a primitive of $h$. Assume that  
\[
\int_J  \psi' g + h \psi =0,
\]
for each $\psi \in C^{\infty}_0(J)$. Then the function $g-H$ is a constant function in $J$. In particular $g \in C^1(J)$.
\end{lemma}

\begin{remark}
\label{rk:PW}
Let $K$ be  a convex body of class $C_+^2$ such that $0\in K$. Following \cite{2020arXiv200704683P} we consider a clockwise-oriented $P$-periodic parameterization $\ga:\rr \to \rr^2$ of $\partial K$.
 For a fixed $v \in \rr$ we take the translated curve $s \to \ga(s+v)-\ga(v)=(x(s),y(s))$  and we consider its horizontal lifting $\Ga_v(s)$ to $\hh^1$ starting at $(0,0,0) \in \hh^1$ for $s=0$, given by
\[
\Ga_v(s)=\left((x(s),y(s), \int_0^s y(\tau) x'(\tau)-x(\tau) y'(\tau) d\tau \right).
\]
The Pansu-Wulff shape associated to $K$ is defined by 
\[
\mathbb{S}_K= \bigcup_{v \in [0,P)} \Ga_v([0,P]).
\]
In \cite[Theorem 3.14]{2020arXiv200704683P} it is shown that the horizontal liftings $\Ga_v$, for each $v \in [0,P)$, are solutions for $H_K= 1$, therefore $\mathbb{S}_K$ has constant prescribed $K$-mean curvature equal to $1$.  Since the curves $\Ga_v$ have the same regularity as $\ptl K$, the $C^2$ regularity result for horizontal curves obtained in Theorem \ref{th:main} is optimal.
\end{remark}

\begin{corollary}
Let $K$ be a $C_+^2$ convex set in $\rr^2$ with $0\in\intt(K)$ and $\norm{\cdot}_K$ the associated left-invariant norm in $\hh^1$.
Let $\Omega \subset \mathbb{H}^1$ be an open set and $E \subset \Omega$ a set of prescribed $K$-mean curvature $f \in C^0(\Omega)$ with $C^1$ boundary $S$. Then the characteristic curves in $S\smallsetminus S_0$ are of class $C^2$.
\end{corollary}
\begin{proof}
Since $S$ is of class $C^1$, in the regular part $S \smallsetminus S_0$ the horizontal normal $\nu_h$  is a nowhere-vanishing continuous vector fields, thus $S \smallsetminus S_0$ is an $\hh$-regular surface. In particular a $C^1$ surface is Lipschitz, thus $S\smallsetminus S_0$ verifies the hypotheses  of Theorem \ref{th:main} and the characteristic curves in $S\smallsetminus S_0$ are of class $C^2$.
\end{proof}
\begin{remark}
When $S$ is of class $C^1$ the proof of Lemma \ref{lm:MC1} is is much easier. Indeed the solution $t_{\eps}$ of the Cauchy Problem \eqref{eq:teps}  is differentiable in $\eps$, thus  the function $\partial t_{\eps}/ \partial \eps$ satisfies the following ODE 
 $$\left(\frac{\partial t_{\eps}}{\partial \eps}\right)'(s)=2 u_t(s, t_{\eps}(s)) \frac{\partial t_{\eps}}{\partial \eps}, \qquad \frac{\partial t_{\eps}}{\partial \eps}(a)=1.$$
 That implies that 
 \[
 \frac{\partial t_{\eps}}{\partial \eps}(s)=e^{\int_a^s  2 u_t(\tau, t_{\eps}(\tau))) d\tau} >0.
 \]
 Since the Jacobian $\mathbf{J}_G$ defined in \eqref{eq:JacG} is equal to $\partial t_{\eps}/ \partial \eps>0$  the change of variables $G(\xi,\eps)$ is invertible. Hence  the rest of  the proof of Lemma \ref{lm:MC1}  goes  in the same way as before. 
\end{remark}

\section{The sub-Finsler mean curvature equation}
\label{sc:KMCE}
Given an Euclidean Lipschitz boundary $S$ whose characteristic curves in $S\smallsetminus S_0$ are of class $C^2$, for each point $p \in S\smallsetminus S_0$ we can define the $K$-mean curvature $H_K$ of $S$ by 
\begin{equation}
\label{eq:meanK}
H_K=\escpr{D_Z \pi_K(\nu_h),Z}=\escpr{\nabla_Z \pi_K(\nu_h),Z},
\end{equation}
where $\nu_h$ is the outer horizontal unit normal to $S$. This definition was given in \cite{2020arXiv200704683P} for surfaces of class $C^2$.

\begin{proposition}
\label{prop:KMCE}
Let $\Omega \subset \mathbb{H}^1$ be an open set and $E \subset \Omega$ a set of prescribed $K$- mean curvature $f \in C^0(\Omega)$ Euclidean Lipschitz and  $\mathbb{H}$-regular boundary $S$. Then $H_K(p)=f(p)$ for each $p\in S\smallsetminus S_0 $.
\end{proposition}
\begin{proof}
By the Implicit Function Theorem for $\hh$-regular surfaces, Theorem~6.5 in \cite{MR1871966}, given a point $p \in S$, after a rotation about the $t$-axis, there exists an open neighborhood $B \subset \mathbb{H}^1 $ of  $p$ such that $B \cap S$ is the intrinsic graph of a function $u: D \to \rr$ where $D$ is a domain in the vertical plane $y=0$. The function $u$ is Euclidean Lipschitz by our assumption. We set $B \cap S= \text{Gr}(u)$. We assume that $E$ is locally the epigraph of $u$.
  
Let $\Gamma(s)$ be a characteristic curve passing through $p$  in $\text{Gr}(u)$ and  $\ga(s)$ its projection on the $xt$-plane. 
The characteristic vector $Z$ defined in \eqref{eq:Z} is given by 
\[
Z=\dfrac{X+(u_x+2uu_t)Y }{(1+(u_x+2uu_t)^2)^{\frac{1}{2}}}.
\]
Since $S$ is $\hh$-regular,  $Z$ and the horizontal unit normal
\[
\nu_h=\dfrac{(u_x+2uu_t)X-Y }{(1+(u_x+2uu_t)^2)^{\frac{1}{2}}}
\] 
are continuous vector fields.
By Lemma \ref{lm:MC1} we have that $M=F(u_x+2u u_t)$ defined in \eqref{eq:M} satisfies the differential equation 
\[
\dfrac{d}{ds} M(\ga(s))=f(\ga(s))
\]
along the characteristic curves.
Therefore  we obtain
\begin{align*}
\dfrac{d}{ds} M(\ga(s))&= F'(u_x+2u u_t) \dfrac{d}{ds} \big[(u_x+2uu_y)(\ga(s))\big] 
\\
&=\dfrac{1}{\phi''(u_x+2u u_t)} \dfrac{d}{ds} \big[(u_x+2uu_y)(\ga(s))\big],
\end{align*}

As in proof of Lemma~\ref{lemmaFK}, we parametrize the lower part of the boundary of the convex body $K$ by a function $\phi$  defined on a closed interval $I \subset \rr$. Again by Lemma~\ref{lemmaFK}  we have
\[
\pi_K(x,-1)=(\varphi(x),\phi(\varphi(x)),
\]
where $\varphi$ is the inverse function of $\phi'$.
Furthermore the $K$-mean curvature defined \eqref{eq:meanK} is equivalent to 
\begin{align*}
H_K&=\escpr{D_Z \pi_K(u_x+2uu_t,-1) , Z}\\
&=\dfrac{\escpr{\dfrac{D}{ds} \big[\varphi(u_x+2uu_t) X_{\ga}+\phi(\varphi(u_x+2uu_t))Y_{\ga}\big], Z}}
{1+(u_x+2uu_t)^2}  \\
&=\dfrac{\varphi'(u_x+2u u_t) \dfrac{d}{ds}(u_x+2uu_t)\big(1+ \phi'(\varphi(u_x+2uu_t)) (u_x+2uu_t)\big)}{1+(u_x+2uu_t)^2}\\
&=\dfrac{1}{\phi''(u_x+2u u_t)} \dfrac{d}{ds} \big[(u_x+2u u_t)(\ga(s))\big].
\end{align*}
Hence we obtain $H_K=\tfrac{d}{ds}M(\ga(s))$ and so $H_K(p)=f(p)$ for each $p\in S\smallsetminus S_0 $.
\end{proof}

The following result allows us to express  the $K$-mean curvature $H_K$ in terms of the sub-Riemannian mean curvature $H_D$.

\begin{proposition}
\label{prop:subKMCE} 
Let $K \subset \rr^2$ be a convex body of class $C_+^2$ such that $0 \in \intt(K)$ and $\pi_K=N_K^{-1}$. Let $\kappa$ be the strictly positive curvature of the boundary $\partial K$. Let $\Omega \subset \mathbb{H}^1$ be an open set and $E \subset \Omega$ a set of prescribed $K$-mean curvature $f \in C^0(\Omega)$ with Euclidean Lipschitz and  $\mathbb{H}$-regular boundary $S$. Then, we have 
\[
H_D(p)=\kappa( \pi_K(\nu_h)) f (p) \qquad \text{for each} \quad p \in S \smallsetminus S_0 ,
\]
where $H_D(p)=\escpr{D_Z \nu_h, Z}$ is the sub-Riemannian mean curvature, $\nu_h$ be the horizontal unit normal at $p$ to $S \smallsetminus S_0 $ and $Z= J(\nu_h)$ be the characteristic vector field.
\end{proposition}
\begin{proof}
By Proposition \ref{prop:KMCE} we have $H_K(p)=f(p)$ for each $p\in S\smallsetminus S_0 $. We remark that Theorem \ref{th:main} implies that $H_K$ is well-defined.

Let $\gamma: (-\eps,\eps) \to S \smallsetminus S_0 $ be the integral  curve of $Z$ passing through $p$, namely  $\ga'(s)=Z_{\ga(s)}$ and $\ga(0)=p$. Let $\nu_h (s)=-J(Z_{\ga(s)})$ be the horizontal unit  normal along $\gamma$ and let 
\[
\pi(\nu_h(s))= \pi_1(\nu_h(s)) X_{\ga(s)} + \pi_2(\nu_h(s)) Y_{\ga(s)}.
\]
Noticing that $\nabla X=\nabla Y=0$
we gain
\[
\dfrac{\nabla}{ds}\Big|_{s=0} \pi(\nu_h(s))= \dfrac{d}{ ds}\Big|_{s=0} \pi_1(\nu_h(s)) X_{\ga(0)} + \dfrac{d}{ ds}\Big|_{s=0} \pi_2(\nu_h(s)) Y_{\ga(0)}.
\]
Setting $\nu_h=a X+ b Y$   we obtain 
\begin{equation}
\label{eq:nablaofpi}
\dfrac{\nabla}{ds}\Big|_{s=0} \pi(\nu_h(s))= (d \pi)_{(a,b)} \left(\dfrac{\nabla}{ds}\Big|_{s=0} \nu_h (s) \right),
\end{equation}
where 
\[
 (d \pi)_{(a,b)}=\left(\begin{array}{cc} \dfrac{\partial \pi_1 }{\partial a}(a,b) & \dfrac{\partial \pi_1}{\partial b} (a,b)\\ [.25 cm] \dfrac{\partial \pi_2}{\partial a} (a,b) & \dfrac{\partial \pi_2}{\partial b} (a,b) \end{array} \right).
\]
Moreover, by Corollary 1.7.3 in \cite{MR3155183} we get $\pi_K= \nabla h$, where $h$ is a $C^2$ function. Thus by Schwarz's theorem  the Hessian $\text{Hess}_{(a,b)}(h)=(d \pi)_{(a,b)}$ is symmetric, i.e. $(d \pi)=(d \pi)^*$. Equation \eqref{eq:nablaofpi} then implies
\[
H_K=\escpr{\nabla_Z \, \pi_K(\nu_h),Z}= \escpr{\nabla_Z \nu_h , (d \pi)_{\nu_h}^* Z}=  \escpr{\nabla_Z \nu_h , (d \pi)_{\nu_h} Z}.
\]
Finally, by Lemma \ref{lm:diff}  we get 
\[
H_K= \dfrac{1}{\kappa(\pi_K(\nu_h))} \escpr{\nabla_Z \nu_h , Z}.
\]
Hence we obtain  $\escpr{D_Z \nu_h , Z}=\kappa( \pi_K(\nu_h)) $, since $D_Z \nu_h=\nabla_Z \nu_h$.
\end{proof}

\begin{lemma}
\label{lm:diff}
Let $K \subset \rr^2$ be a convex body of class $C_+^2$ such that $0 \in \intt(K)$ and $N_K$ be the Gauss map of $\partial K$. Let $\kappa$ be the strictly positive curvature of the boundary $\partial K$. Let $S$ be an $\hh$-regular surface with horizontal unit normal $\nu_h$ and characteristic vector field $Z= J(\nu_h)$. Then we have
\[
(d \pi)_{\nu_h} Z= \dfrac{1}{\kappa} Z \quad \text{and} \quad  (d \pi)_{\nu_h} \nu_h=0,
\]
where $(d \pi)_{\nu_h} $ is the differential of $\pi_K= N_K^{-1}$.
\end{lemma}

\begin{proof}
Let $\alpha(t)=(x(t),y(t))$ be an arc-length parametrization of $\partial K$ such that $\dot{x}^2(t)+ \dot{y}^2(t) =1$. Let $\nu_h=a X+ b Y$ be the horizontal unit normal to $S$, with $a=\cos(\theta)$ and $b=\sin(\theta)$ and $\theta \in (-\tfrac{\pi}{2},  \tfrac{\pi}{2})$. Notice that $\theta= \arctan(\tfrac{b}{a})$. Then we have  
\[
\pi_K(a,b)= N_K^{-1} ((a,b)).
\]
Let $\varphi:(-\tfrac{\pi}{2}, \tfrac{\pi}{2}) \to \rr $ be the function satisfying
\[
\pi_K(\cos(\theta), \sin(\theta))=(x (\varphi(\theta)), y(\varphi(\theta))).
\]
If we consider the normal $N_K$ of the previous equality we obtain
\[
(\cos(\theta), \sin(\theta))=(\dot{y}(\varphi(\theta)), -\dot{x} (\varphi(\theta))).
\]
Therefore we have 
\[
\theta=\arctan\left(-\dfrac{\dot{x}}{\dot{y}} (\varphi(\theta))\right)
\]
for each $\theta \in (-\tfrac{\pi}{2}, \tfrac{\pi}{2}) $. That means that $\varphi$ is the  inverse of the function $
\arctan(-\tfrac{\dot{x}}{\dot{y}} (t))$, that is invertible since 
\[
\dfrac{d}{dt}\arctan(-\tfrac{\dot{x}}{\dot{y}} (t))=\dot{x} \ddot{y}- \dot{y} \ddot{x}= \kappa(t)>0.
\]
Let $Z=J(\nu_h)=-bX+aY$ be the characteristic vector field, then we have $(d \pi)_{(a,b)}=(d \pi)_{(a,b)}^*$ and
\[
(d \pi)_{(a,b)}^* Z= \left( \begin{array}{c} -b \dfrac{\partial \pi_1 }{\partial a} + a\dfrac{\partial \pi_2 }{\partial a}\\ [.25 cm] -b \dfrac{\partial \pi_1 }{\partial b}+a\dfrac{\partial \pi_2 }{\partial b} \\ \end{array}\right),
\]
where 
\begin{align*}
\pi_1(a,b)= x(\varphi(\arctan(\tfrac{b}{a}))), \qquad \pi_2(a,b)=y(\varphi(\arctan(\tfrac{b}{a}))).
\end{align*}
Thus we get
\begin{align*}
(d \pi)_{(a,b)}^* Z=\varphi'(\arctan(\tfrac{b}{a}))) Z= \dfrac{1}{\kappa(\varphi(\theta))} Z.
\end{align*}
A similar straightforward computation shows that $(d \pi)_{\nu_h} \nu_h=0$.
\end{proof}

\bibliography{sub-finsler}

\begin{thebibliography}{10}

\bibitem{MR2223801}
L.~Ambrosio, F.~Serra~Cassano, and D.~Vittone.
\newblock Intrinsic regular hypersurfaces in {H}eisenberg groups.
\newblock {\em J. Geom. Anal.}, 16(2):187--232, 2006.

\bibitem{MR2333095}
V.~Barone~Adesi, F.~Serra~Cassano, and D.~Vittone.
\newblock The {B}ernstein problem for intrinsic graphs in {H}eisenberg groups
  and calibrations.
\newblock {\em Calc. Var. Partial Differential Equations}, 30(1):17--49, 2007.

\bibitem{MR1511684}
L.~E.~J. Brouwer.
\newblock Beweis des ebenen {T}ranslationssatzes.
\newblock {\em Math. Ann.}, 72(1):37--54, 1912.

\bibitem{MR2583494}
L.~Capogna, G.~Citti, and M.~Manfredini.
\newblock Regularity of non-characteristic minimal graphs in the {H}eisenberg
  group {$\Bbb H^1$}.
\newblock {\em Indiana Univ. Math. J.}, 58(5):2115--2160, 2009.

\bibitem{MR2312336}
L.~Capogna, D.~Danielli, S.~D. Pauls, and J.~T. Tyson.
\newblock {\em An introduction to the {H}eisenberg group and the
  sub-{R}iemannian isoperimetric problem}, volume 259 of {\em Progress in
  Mathematics}.
\newblock Birkh{\"a}user Verlag, Basel, 2007.

\bibitem{MR3794892}
J.-H. Cheng, H.-L. Chiu, J.-F. Hwang, and P.~Yang.
\newblock Umbilicity and characterization of {P}ansu spheres in the
  {H}eisenberg group.
\newblock {\em J. Reine Angew. Math.}, 738:203--235, 2018.

\bibitem{MR2165405}
J.-H. Cheng, J.-F. Hwang, A.~Malchiodi, and P.~Yang.
\newblock Minimal surfaces in pseudohermitian geometry.
\newblock {\em Ann. Sc. Norm. Super. Pisa Cl. Sci. (5)}, 4(1):129--177, 2005.

\bibitem{MR2262784}
J.-H. Cheng, J.-F. Hwang, and P.~Yang.
\newblock Existence and uniqueness for {$p$}-area minimizers in the
  {H}eisenberg group.
\newblock {\em Math. Ann.}, 337(2):253--293, 2007.

\bibitem{MR2481053}
J.-H. Cheng, J.-F. Hwang, and P.~Yang.
\newblock Regularity of {$C^1$} smooth surfaces with prescribed {$p$}-mean
  curvature in the {H}eisenberg group.
\newblock {\em Math. Ann.}, 344(1):1--35, 2009.

\bibitem{MR2235475}
G.~Citti and A.~Sarti.
\newblock A cortical based model of perceptual completion in the
  roto-translation space.
\newblock {\em J. Math. Imaging Vision}, 24(3):307--326, 2006.

\bibitem{MR2354992}
D.~Danielli, N.~Garofalo, and D.~M. Nhieu.
\newblock Sub-{R}iemannian calculus on hypersurfaces in {C}arnot groups.
\newblock {\em Adv. Math.}, 215(1):292--378, 2007.

\bibitem{MR2386783}
D.~Danielli, N.~Garofalo, and D.-M. Nhieu.
\newblock A partial solution of the isoperimetric problem for the {H}eisenberg
  group.
\newblock {\em Forum Math.}, 20(1):99--143, 2008.

\bibitem{MR2472175}
D.~Danielli, N.~Garofalo, D.~M. Nhieu, and S.~D. Pauls.
\newblock Instability of graphical strips and a positive answer to the
  {B}ernstein problem in the {H}eisenberg group {$\Bbb H^1$}.
\newblock {\em J. Differential Geom.}, 81(2):251--295, 2009.

\bibitem{MR62214}
E.~De~Giorgi.
\newblock Su una teoria generale della misura {$(r-1)$}-dimensionale in uno
  spazio ad {$r$} dimensioni.
\newblock {\em Ann. Mat. Pura Appl. (4)}, 36:191--213, 1954.

\bibitem{MR2214654}
S.~Dragomir and G.~Tomassini.
\newblock {\em Differential geometry and analysis on {CR} manifolds}, volume
  246 of {\em Progress in Mathematics}.
\newblock Birkh\"{a}user Boston, Inc., Boston, MA, 2006.

\bibitem{Evans2015}
L.~C. Evans and R.~F. Gariepy.
\newblock {\em Measure theory and fine properties of functions}.
\newblock Textbooks in Mathematics. CRC Press, Boca Raton, FL, revised edition,
  2015.

\bibitem{MR3412408}
V.~Franceschi, G.~P. Leonardi, and R.~Monti.
\newblock Quantitative isoperimetric inequalities in {$\Bbb{H}^n$}.
\newblock {\em Calc. Var. Partial Differential Equations}, 54(3):3229--3239,
  2015.

\bibitem{monti-finsler}
V.~Franceschi, R.~Monti, A.~Righini, and M.~Sigalotti.
\newblock {The isoperimetric problem for regular and crystalline norms in
  $\mathbb{H}^1$}.
\newblock \href{https://arxiv.org/abs/2007.11384}{arXiv:2007.11384}, 22 Jul
  2020.

\bibitem{MR1871966}
B.~Franchi, R.~Serapioni, and F.~{Serra Cassano}.
\newblock Rectifiability and perimeter in the {H}eisenberg group.
\newblock {\em Math. Ann.}, 321(3):479--531, 2001.

\bibitem{MR3474402}
M.~Galli.
\newblock The regularity of {E}uclidean {L}ipschitz boundaries with prescribed
  mean curvature in three-dimensional contact sub-{R}iemannian manifolds.
\newblock {\em Nonlinear Anal.}, 136:40--50, 2016.

\bibitem{MR3412382}
M.~Galli and M.~Ritor{\'e}.
\newblock Regularity of {$C^1$} surfaces with prescribed mean curvature in
  three-dimensional contact sub-{R}iemannian manifolds.
\newblock {\em Calc. Var. Partial Differential Equations}, 54(3):2503--2516,
  2015.

\bibitem{MR1404326}
N.~Garofalo and D.-M. Nhieu.
\newblock Isoperimetric and {S}obolev inequalities for
  {C}arnot-{C}arath{\'e}odory spaces and the existence of minimal surfaces.
\newblock {\em Comm. Pure Appl. Math.}, 49(10):1081--1144, 1996.

\bibitem{MR4118581}
G.~Giovannardi.
\newblock Higher dimensional holonomy map for rules submanifolds in graded
  manifolds.
\newblock {\em Anal. Geom. Metr. Spaces}, 8(1):68--91, 2020.

\bibitem{2019arXiv190505131}
G.~Giovannardi, G.~Citti, and M.~Ritoré.
\newblock Variational formulas for submanifolds of fixed degree.
\newblock \href{https://arxiv.org/abs/1905.05131}{arXiv:1905.05131}, 13 May
  2019.

\bibitem{MR1201446}
P.~Haj{\l}asz.
\newblock Change of variables formula under minimal assumptions.
\newblock {\em Colloq. Math.}, 64(1):93--101, 1993.

\bibitem{MR2401420}
R.~K. Hladky and S.~D. Pauls.
\newblock Constant mean curvature surfaces in sub-{R}iemannian geometry.
\newblock {\em J. Differential Geom.}, 79(1):111--139, 2008.

\bibitem{MR3048517}
R.~K. Hladky and S.~D. Pauls.
\newblock Variation of perimeter measure in sub-{R}iemannian geometry.
\newblock {\em Int. Electron. J. Geom.}, 6(1):8--40, 2013.

\bibitem{MR2609016}
A.~Hurtado, M.~Ritor{\'e}, and C.~Rosales.
\newblock The classification of complete stable area-stationary surfaces in the
  {H}eisenberg group {$\Bbb H^1$}.
\newblock {\em Adv. Math.}, 224(2):561--600, 2010.

\bibitem{MR2358000}
A.~Hurtado and C.~Rosales.
\newblock Area-stationary surfaces inside the sub-{R}iemannian three-sphere.
\newblock {\em Math. Ann.}, 340(3):675--708, 2008.

\bibitem{MR2177813}
G.~P. Leonardi and S.~Masnou.
\newblock On the isoperimetric problem in the {H}eisenberg group {${\Bbb
  H}^n$}.
\newblock {\em Ann. Mat. Pura Appl. (4)}, 184(4):533--553, 2005.

\bibitem{MR2000099}
G.~P. Leonardi and S.~Rigot.
\newblock Isoperimetric sets on {C}arnot groups.
\newblock {\em Houston J. Math.}, 29(3):609--637, 2003.

\bibitem{MR2402213}
R.~Monti.
\newblock Heisenberg isoperimetric problem. {T}he axial case.
\newblock {\em Adv. Calc. Var.}, 1(1):93--121, 2008.

\bibitem{MR2548252}
R.~Monti and M.~Rickly.
\newblock Convex isoperimetric sets in the {H}eisenberg group.
\newblock {\em Ann. Sc. Norm. Super. Pisa Cl. Sci. (5)}, 8(2):391--415, 2009.

\bibitem{2008.04027}
S.~Nicolussi and M.~Ritoré.
\newblock {Area-minimizing cones in $\mathbb{H}^1$}.
\newblock \href{https://arxiv.org/abs/2008.04027}{arXiv:2008.04027}, 10 Aug
  2020.

\bibitem{MR3984100}
S.~Nicolussi and F.~Serra~Cassano.
\newblock The {B}ernstein problem for {L}ipschitz intrinsic graphs in the
  {H}eisenberg group.
\newblock {\em Calc. Var. Partial Differential Equations}, 58(4):Paper No. 141,
  28, 2019.

\bibitem{MR676380}
P.~Pansu.
\newblock Une in{\'e}galit{\'e} isop{\'e}rim{\'e}trique sur le groupe de
  {H}eisenberg.
\newblock {\em C. R. Acad. Sci. Paris S{\'e}r. I Math.}, 295(2):127--130, 1982.

\bibitem{MR2225631}
S.~D. Pauls.
\newblock {$H$}-minimal graphs of low regularity in {$\Bbb H^1$}.
\newblock {\em Comment. Math. Helv.}, 81(2):337--381, 2006.

\bibitem{2020arXiv200704683P}
J.~{Pozuelo} and M.~{Ritor{\'e}}.
\newblock {Pansu-Wulff shapes in $\mathbb{H}^1$}.
\newblock \href{https://arxiv.org/abs/2007.04683}{arXiv:2007:04683}, 9 Jul
  2020.

\bibitem{MR2448649}
M.~Ritor\'{e}.
\newblock Examples of area-minimizing surfaces in the sub-{R}iemannian
  {H}eisenberg group {$\Bbb H^1$} with low regularity.
\newblock {\em Calc. Var. Partial Differential Equations}, 34(2):179--192,
  2009.

\bibitem{MR2898770}
M.~Ritor{\'e}.
\newblock A proof by calibration of an isoperimetric inequality in the
  {H}eisenberg group {${\Bbb H}^n$}.
\newblock {\em Calc. Var. Partial Differential Equations}, 44(1-2):47--60,
  2012.

\bibitem{MR2435652}
M.~Ritor{\'e} and C.~Rosales.
\newblock Area-stationary surfaces in the {H}eisenberg group {$\Bbb H^1$}.
\newblock {\em Adv. Math.}, 219(2):633--671, 2008.

\bibitem{snchez2017subfinsler}
A.~P. S{\'a}nchez.
\newblock Sub-finsler heisenberg perimeter measures, 2017.

\bibitem{MR3155183}
R.~Schneider.
\newblock {\em Convex bodies: the {B}runn-{M}inkowski theory}, volume 151 of
  {\em Encyclopedia of Mathematics and its Applications}.
\newblock Cambridge University Press, Cambridge, expanded edition, 2014.

\bibitem{MR2961944}
G.~Teschl.
\newblock {\em Ordinary differential equations and dynamical systems}, volume
  140 of {\em Graduate Studies in Mathematics}.
\newblock American Mathematical Society, Providence, RI, 2012.

\end{thebibliography}

\end{document}